\documentclass[12pt,reqno,twoside]{amsart}
\usepackage{amsmath}
\usepackage{amssymb}
\usepackage{epsfig}
\usepackage{color}

\numberwithin{equation}{section}

\title{MULTIPLICATIVE DIOPHANTINE EXPONENTS OF HYPERPLANES}

\author{Yuqing Zhang}

\address{Brandeis University, Waltham MA
02454-9110 {\tt yqzhang@brandeis.edu}}

\newif\ifdraft\drafttrue

\draftfalse

\newcommand{\R}{{\mathbb{R}}}

\newcommand{\Z}{{\mathbb{Z}}}

\newcommand{\N}{{\mathbb{N}}}

\newcommand{\q}{{\bf{q}}}

\newcommand{\GL}{\operatorname{GL}}

\newcommand{\SL}{\operatorname{SL}}

\newcommand{\z}{{\bf z}}
\newcommand{\w}{{\bf w}}

\newcommand{\p}{{\bf p}}
\newcommand{\y}{{\bf y}}

\newcommand {\ignore}[1]  {}

\newtheorem{thm}{Theorem}[section]

\begin{document}

\begin{abstract}

We study multiplicative Diophantine approximation property of vectors and compute Diophantine exponents of hyperplanes via dynamics.

\end{abstract}


\maketitle

\section{Introduction}

For a vector $\y  = (y_1, \ldots, y_n) \in \R^n$ its Diophantine
exponent is defined by
\begin{equation}
\label{eq: des}\omega(\y)=\textrm{sup} \{v|\textrm{ } \exists
\infty \textrm{ many } \q=(q_1,q_2,\ldots,q_n) \in \Z^n \textrm{ with }
\nonumber
\end{equation}
\begin{equation}
|\q\y+p|=|q_1y_1+q_2y_2+\ldots+q_ny_n+p|<\|\q\|^{-v}\textrm{ for some } p \in \Z\}
\end{equation}

In this standard definition the size of the vector $\q$ is measured
by its supreme norm and all its entries with smaller absolute values
than the norm  are simply ignored. A more elaborate way of measuring
$\q$ is by

\begin{displaymath}
\prod_\times(\q)=\prod_{i=1\ q_i\neq 0}^{n}|q_i|
\end{displaymath}
And we have the \textit{multiplicative} version of Diophantine
exponent for $\y$:
\begin{equation}
\label{eq: dem} \omega^\times(\y)=\textrm{sup} \{v|\textrm{ }
\exists \infty \textrm{ many } \q \in \Z^n \textrm{ with } |
\q\y+p|<{\prod_\times(\q)}^{-v/n} \textrm{ for some } p \in \Z\}
\end{equation}

Besides, $\y$ can be viewed as a column vector and be equipped with the $\sigma$ version of
Diophantine exponent:
\begin{displaymath}
\sigma(\y)=\textrm{sup} \Big\{v|\textrm{ }
\exists \infty \textrm{ many } q \in \Z \textrm{ with } \left\|  \begin{array}{c}
qy_1+p_1\\\vdots  \\ qy_n+p_{n}\\ \end{array}  \right\|<|q|^{-v}
\textrm{ for some } \p \in \Z^n\ \Big\}
\end{displaymath}

 We further define the \textit{multiplicative} Diophantine exponent
$\omega^\times(\mu)$ of a
 Borel measure $\mu$ to be the $\mu$-essential supreme of the $\omega^\times$
function, that is,
\begin{equation}
\label{eq: demu} \omega^\times(\mu)=\textrm{sup} \{v|\textrm{ }
\mu\{\y|\textrm{ }\omega^\times(\y)>v\}>0 \}
\end{equation}

And for a smooth submanifold M of $\R^n$ with measure class of its
Riemannian volume denoted by $\mu$, we set
$\omega^\times(M)=\omega^\times(\mu)$. $\omega(\mu)$ and $\omega(M)$
are defined in an analogous way.\\ From these definitions a couple
of inequalities are derived:
\begin{equation}
\label{eq: sm}\omega^\times(\y)\geqq \omega(\y)
\end{equation}

\begin{equation}
\label{eq: sm1}\omega^\times(M)\geqq \omega(M)
\end{equation}

Suppose $L$ is an affine subspace of $\R^n$. $\omega(L)$ has
been studied to great lengths, both by elementary methods  as in [J]
and by quantitative nondivergence by D. Kleinbock in [K1][K3].

By comparison the multiplicative exponent of $L$ is inherently more
complicated. Here advanced mathematical tools appear to be much more
desirable, if not indispensable.In this paper we apply nondivergence which has been developed and
strengthened in [K1][K2][K3] to find out multiplicative exponents of
hyperplanes and their nondegenerate submanifolds.

One of the major theorems we are to establish is:
\begin{thm}\label{cor: dicho}
Let $L$ be a hyperplane of  $\R^n$ parameterized by \\ $(x_1, \ldots,
x_{n-1}, a_1x_1+\ldots+a_{s-1}x_{s-1}+b)$ with $a_i\neq0,
i=1,\ldots,s-1$,$s\leqslant n$ and $M$ a non-degenerate
submanifold in $L$. We have\\
 $ \textrm{
 }\omega^\times(L)=\omega^\times(M)=max\big\{n,\dfrac{n}{s}\sigma(a_1,\ldots,a_{s-1},b)\big\}$

\end{thm}

\section{Quantitative  nondivergence and proof of Theorem 1.1}
\newtheorem{defi}{Definition}[section]
\begin{defi}
$W^\times_v=\{\y \in R^n \textrm{  }|\omega^\times(\y)\geqq v\}$
\end{defi}

\begin{defi}
$\Omega_{n+1}=\SL(n+1,\R)\diagup \SL(n+1,\Z)$
\end{defi}
 $\Omega_{n+1}$ is non-compact,   and
\begin{equation} \label{eq: space}
\Omega_{n+1}=\bigcup_{\epsilon>0}K_\epsilon
\end{equation}
where  $K_\epsilon=\{\Lambda \in \Omega_{n+1}| \textrm{  }\|v\| \geq
\epsilon \textrm{  for all nonzero   }v \in \Lambda\}$. Each
$K_\epsilon$ is compact. \\
We associate $\y$ with a $(n+1)\times(n+1)$ matrix
\begin{displaymath}
u_\y=\left(\begin{array}{cc} 1 & \y\\
0 & I_n\end{array}\right)
\end{displaymath}
 And for $\textbf{t}=(t_1,\ldots,t_n)$ with $t_i\geqslant0$, set
 $t=\sum_{i=1}^{n}t_i$,
$g_\textbf{t}=diag(e^t,e^{-t_1},\ldots,e^{-t_n})$

The following lemma will enable us to study multiplicative
diophantine approximation via $g_\textbf{t}$ action. It is
essentially the same as Lemma 2.1 of [K2] and a proof can be found
there.

\newtheorem{lemma}[defi]{Lemma}
\begin{lemma}
Suppose we are given a positive integer $k$ and a set $E$ of $(x,\z)
\in \R^{n+1}$ which is discrete and homogeneous with respect to
positive integers. $z_i\geqslant 1$ for $i_1,\ldots, i_k$ and
$z_i=0$ for the rest $i$. Take $v>n$ and
$c_k=\dfrac{v-n}{kv+n}$,then the following are equivalent:

\begin{enumerate}
\item $\exists (x,\z) \in E$  with arbitrarily large $\|\z\|$ such that
$|x| \leq {\prod_\times(\z)}^{-v/n}$ \item $\exists$ an unbounded
set of $\textbf{t}\in \R_+^n$ such
that for some $(x,\z) \in E \backslash\{0\}$ one has\\
max$(e^{t}|x|,e^{-t_i}|z_i|) \leqq e^{-c_kt}$

\end{enumerate}

\end{lemma}
Accordingly $Z^{n+1}$ is decomposed as $\bigcup_{k=0}^n \Z^{n+1}_{k}$, where
\begin{equation}
 \Z^{n+1}_{k}=\big\{(p, q_1,q_2,\ldots,q_n)\in
\Z^{n+1}\big|\textrm{ exactly k entries of }(q_1,q_2,\ldots,q_n)\textrm{ nonzero}\big\}
\end{equation}

Take $v>n$, $\y  \in \R^n$ and $E=\{(|\q\y+p|,\q )\textrm{ }
|\textrm{  }p \in \R, \q \in \R^n\}$ we see

(2.3.1) is equivalent to $\omega^\times(\y) \geqq v$ and (2.3.2) equivalent to
\begin{equation} \label{eq: gt} g_\textbf{t}u_\y\Z^{n+1}_k \textrm{ contains at least one nonzero vector with norm } \leq e^{-c_kt} \nonumber
\end{equation}
\begin{equation}
\textrm{ for an unbounded set of } \textbf{t} \in \R_+^{n+1}
\end{equation}
For any fixed $v>n$, $c_1>c_2>\ldots>c_n$.\\
 $c_k=\dfrac{v-n}{kv+n}
\Leftrightarrow  v=\dfrac{n+nc_k}{1-kc_k}$

If $\gamma_k(\y)=$ sup \{$c_k |\textrm{  } (2.3) \textrm{ holds  }\} $then by preceding lemma
\begin{equation}\label{eq: omegay}
\omega^\times(\y)=max_{\substack{1\leqslant k \leqslant n}}\dfrac{n+n\gamma_k(\y)}{1-k\gamma_k(\y)}
\end{equation}

Suppose $\lambda$ is a measure on $\R^n$,  and $v\geq n$ , by
definition $\omega^\times(\lambda)\leq v$ iff $\lambda(W^\times_u)=0$ for
any $u>v$.By Borel-Cantelli Lemma, a sufficent condition for $\omega^\times(\lambda)\leq v$ is
\begin{equation}\label{eq: bcl}
\sum_{t=1}^\infty\lambda(\{\y|g_\textbf{t}u_\y\Z^{n+1}_k \textrm{ contains at least one nonzero vector with norm} \leq e^{-d_kt}\})<\infty \nonumber
\end{equation}
\begin{equation}
\forall d_k>c_k, \quad 1\leq k\leq n
\end{equation}

(2.6) provides one way of determining the upper bounds of $\omega^\times(\lambda)$.To make
it more explicit,quantitative nondivergence is needed.

\begin{lemma}\label{lemma:  thm2.2}
Let $k$, $N$ $\in \N$ and $C,D,\alpha,\rho >0$ and suppose we are
given an $N$-Besicovitch metric space $X$, a ball $B=B(x_0,
r_0)\subset X$, a measure $\mu$ which is $D$-Federer on
$\tilde{B}=B(x_0, 3^kr_0)$ and a map $h$: $\tilde{B}\rightarrow
\GL_k(\R)$. Assume the following two conditions hold:\\
\begin{enumerate}
\item $\forall \quad \Gamma \subset \Z^{k}$,  the function $ x\rightarrow
\|h(x)\Gamma\|$ is $(C, \alpha)$-good on $\tilde{B}$ with respect to
$\mu$;

\item $\forall \quad \Gamma \subset \Z^{k}$, $\|h(\cdot )\Gamma\|_{\mu, B}\geq \rho^{rk(\Gamma)}$
\end{enumerate}
Then for any positive $\epsilon \leq \rho$ one has
\begin{equation}
\mu(\{x \in B|\textrm{ } h(x)\Z^{k} \notin K_{\epsilon}\})\leq
kC(ND^2)^{k}(\dfrac{\epsilon}{\rho})^{\alpha}\mu(B)
\end{equation}
\end{lemma}

Lemma 2.4 implies the following proposition:
\newtheorem{proposition}[defi]{Proposition}
\begin{proposition}
Let $X$ be a Besicovitch metric space, $B=B(x, r)\subset X$, $\mu$ a
measure which is D-Federer on $\tilde{B}=B(x, 3^{n+1}r)$ for some
$D>0$ and $f$ a continuous map from $\tilde{B}$ to $\R^n$. Take
$v > n$, $c_k=\dfrac{v-n}{kv+n},k=1,\ldots,n$  and assume that
\begin{enumerate}
\item
$\exists c,\alpha >0$ such that all the functions   $x\rightarrow
\|g_tu_{f(x)}\Gamma\|$, $\Gamma \subset \Z^{n+1}$ are $(c, \alpha)$-
good on  $\tilde{B}$ with respect to $\mu$

\item
for any $d_k>c_k$, $\exists T=T(d_k)>0$ such that for any $t\geqq T$ and
 any $\Gamma \subset \Z^{n+1}$ one has $\|g_\textbf{t}u_{f(\cdot)}\Gamma\|_{\mu,
 B}\geq e^{-rk(\Gamma)d_kt}$

\end{enumerate}
Then $\omega^\times(f_*(\mu|_B))\leq v$.
\end{proposition}
\begin{proof}
We set $\lambda=f_*(\mu|B)$. (2.4.1) is the same as (2.5.1).
(2.5.2) implies (2.4.2) for any $t>T(\dfrac{c_k+d_k}{2})$. Hence by
lemma 2.4, \\
$\lambda(\{\y|g_\textbf{t}u_\y\Z^{n+1}_k \textrm{ contains at least one nonzero vector with norm} \leq e^{-d_kt}\})\leq
\lambda(\{\y|\textrm{  }g_tu_\y\Z^ {n+1} \notin
K_{e^{-d_kt}}\})
=\mu(\{x \in B |\textrm{  }h(x)\Z^ {n+1} \notin
K_{e^{-d_kt}}\}$
$\leq const \cdot e^{-\alpha\frac{d_k-c_k}{2}t}$ for all but finitely
many $t \in \N$. Therefore (2.5) follows.
\end{proof}
\bigskip

Suppose $\R^{n+1}$ has standard basis $e_1,\ldots,e_{n+1}$ and we extend its Euclidiean structure to $\bigwedge^j(\R^{n+1})$, then for all index sets $I \subseteq \{1,2,\ldots,{n+1}\}$,
$\{e_I\big|e_I=e_{i_1}\wedge\ldots\wedge e_{i_j},\sharp I=j\}$ form an orthogonal basis of $\bigwedge^j(\R^{n+1})$.
We identify $\Gamma$, subgroup of $\Z^{n+1}$ of rank $j$ with $\w \in \bigwedge^j(\R^{n+1})$ and reproduce the calculations
done in [K3].

\begin{displaymath}
g_tu_\y\w=\sum_{\substack{I}}e^{-\sum_{\substack{i \in I}}t_i}\langle e_I,\w \rangle e_I+
\sum_{\substack{J}}e^{t-\sum_{\substack{i \in J}}t_i}(\sum_{i=0}^n\langle e_i\wedge e_J,\w \rangle y_i) e_0 \wedge e_J
\end{displaymath}

\begin{displaymath}
\widetilde{f}=(1,f_1,\ldots,f_n)=(1,g_1,\ldots,g_s)R
\end{displaymath}
for $(1,g_1,\ldots,g_s)$ linearly independent and $R$ a $(s+1)\times(n+1)$ matrix, set

\begin{displaymath}
C_J(\w)=\langle e_i\wedge e_J,\w \rangle y_i) e_J, \sharp J=j-1
\end{displaymath}

Up to some constant
\begin{equation}
\big\|g_tu_f \w\big\|=max\Big(e^{-\sum_{\substack{i \in I}}t_i}\big\|\langle e_I,\w \rangle\big\|,\quad e^{t-\sum_{\substack{i \in J}}t_i}\big\|RC_J(\w)\big\|
\Big)
\end{equation}
(2.5.2) can be rewritten as

$\forall d_k>c_k, \exists T=T(d_k)>0 \textrm { such that for any }t\geqq T \textrm { and
 any }
 \Gamma \in \Z^{n+1} \textrm{ of rank j one has} $
\begin{equation}
max\Big(e^{-\sum_{\substack{i \in I}}t_i}\big\|\langle e_I,\w \rangle\big\|,\quad e^{t-\sum_{\substack{i \in J}}t_i}\big\|RC_J(\w)\big\|
\Big)\geqq e^{-jd_kt},1\leqq j\leqq n
\end{equation}
Note that $j$ and $k$ are two independent variables: $j$ denotes rank of subgroup of $\Z^{n+1}$ and $k$ the number of
nonzero entries in the vectors used for approximation. \\

For hyperplane $L$ parameterized in Theorem 1.1,
\begin{equation}
R=\left( I_n  \begin{array}{c}
b\\a_1\\ \vdots \\  a_{s-1} \\0\\ \vdots \\0 \\ \end{array} \right)
\end{equation}
$a_i \neq 0, 1\leq i \leq s-1$

\begin{equation}
f=(1,x_1,\ldots,x_{n-1},a_1x_1+\ldots a_{s-1}x_{s-1}+b), a_i \neq 0,1\leqslant i\leqslant s-1
\end{equation}
Thanks to Lemma 4.6 of [K2] we know that $\|RC_J(\w)\|\geq 1$ for $j>1$, so  (2.8) is automatically fulfilled for such
subgroups of $Z^{n+1}$ .We only need to check for subgroups of rank 1, or vectors, for a negation of (2.8).
\begin{displaymath}
\w=(p_0,p_1,\ldots,p_n) \in \Z^{n+1}
\end{displaymath}
\begin{displaymath}
\big\|RC_J(\w)\big\|=\left\|  \begin{array}{c}
p_0+bp_n\\p_1+a_1p_n\\ p_2+a_2p_n \\ \vdots \\p_{s-1}+a_{s-1}p_n \\p_s \\ \vdots \\ p_{n-1}\\ \end{array}  \right\|
\end{displaymath}
To avoid $\|RC_J(\w)\|\geq 1$,$p_s,p_{s+1},\ldots,p_{n-1}$ must all be zero. Since  $a_i \neq 0,1\leqslant i\leqslant s-1$,$p_1,\ldots,p_{s-1}$ and
$p_n$ must all be nonzero. By previous notation
$\w=(p_0,p_1,\ldots,p_n)\in \Z_s^{n+1}$. And
\begin{displaymath}
\big\|RC_J(\w)\big\|=\left\|  \begin{array}{c}
p_0+bp_n\\p_1+a_1p_n\\ p_2+a_2p_n \\ \vdots \\p_{s-1}+a_{s-1}p_n \end{array}  \right\|
\end{displaymath}
The above observations coupled with (2.8) supply a handy tool for establishing upper bounds of multiplicative exponents of hyperplanes:
\begin{proposition}
Let $X$ be a Besicovitch metric space, $B=B(x, r)\subset X$, $\mu$ a
measure which is D-Federer on $\tilde{B}=B(x, 3^{n+1}r)$ for some
$D>0$ and $f$ a continuous map from $\tilde{B}$ to $\R^n$ defined in (2.10). Take
$v > n$, $c_s=\dfrac{v-n}{sv+n}$  and assume that
\begin{enumerate}
\item
$\exists c,\alpha >0$ such that all the functions   $x\rightarrow
\|g_tu_{f(x)}\Gamma\|$, $\Gamma \subset \Z^{n+1}$ are $(c, \alpha)$-
good on  $\tilde{B}$ with respect to $\mu$

\item
for any $d_s>c_s$, $\exists T=T(d_s)>0$ such that for any $t\geqq T$ and
 any $\w \in \Z^{n+1}_s$ one has
 \begin{displaymath}
 max\Big(e^{-t_i}\big\|\langle e_i,\w \rangle\big\|,\quad e^{t}\big\|RC_J(\w)\big\|
\Big)\geqq e^{-d_st}
 \end{displaymath}
\end{enumerate}
Then $\omega^\times(f_*(\mu|_B))\leq v$.
\end{proposition}
\begin{lemma}\label{lemma:  negation}
Let $\mu$ be a measure on a set $B$, take $v>n$ and
$c_s=\frac{v-n}{sv+n}$ and let $f$ as in (2.10)  be a map  such
that $\mathrm{(2.6.2)}$  does not hold. Then $f(B\bigcap supp
\textrm{ } \mu)\subset W_u^\times$ for some $u>v$.
\end{lemma}
\begin{proof}

The assumption of the lemma says that $\exists$ an unbounded
sequence $\mathbf{t}_i$ and a sequence of $\w \in Z^{n+1}_{s}$  such that
$\forall x \in B\bigcap supp \textrm{ } \mu \quad$\\$
\|g_\mathbf{ti}u_{f(x)}\w\|<e^{-d_st_i}$.\\
Hence $g_{\mathbf{ti}}u_{f(x)}\Z^{n+1}_s $  has at least one nonzero vector with norm $<e^{-d_st_i} $ $\Rightarrow \gamma_s(f(x))\geqq d_s$. \quad
By (2.4) the lemma holds.
\end{proof}

\newtheorem{thrr}[defi]{Theorem}
\begin{thm}
Let $\mu$ be a Federer measure on a Besicovitch metric space $X$,$L$  an affine subspace of $\R^n$ as in Theorem 1.1.
Let $f:X\rightarrow L$ be a continuous map which is $\mu-$good and $\mu-$nonplanar in L described in (2.10). Then the following are equivalent
for $v>n$
\begin{enumerate}

\item $\{x \in$ supp $\mu |f(x) \notin W^\times_u\}$ is nonempty for any $u>v$
\item $w^\times(f_*\mu)\leqslant v$
\item (2.6.2) holds for $R$ of (2.9)
\end{enumerate}
\end{thm}
\begin{proof}
Suppose (2.8.2) holds then $\{x \in$ supp $\mu |f(x) \notin W^\times_u\}$ has full measure for any $u>v$ hence (2.8.1) holds

If (2.8.3) holds then (2.6.2) implies (2.5.2) for (2.10) map by previous discussion concerning subgroups of various ranks. According to Lemma 1.1 of [K3] (2.5.1) is met.
Proposition(2.5) is therefore applicable to establish (2.8.2).

If (2.8.3) does not hold , no ball $B$ intersecting  supp $\mu$ satisfies (2.6.2). By Lemma 2.7 $f(B\bigcap$ supp $\mu) \subset
W_u^\times$ for some $u>v$. This contradicts (2.8.1).

\end{proof}

The above theorem shows that if $\exists y \in L$ with $\omega^\times(y)\geqslant v$ then the set
$\{y \in L| \omega^\times(y)\geqslant v\}$ has full measure. And we have the theorem for nondegenerate submanifold:
\begin{thm}
Let $M$ be a nondegenerate submanifold of $L$ as in Theorem 1.1\\
\begin{displaymath}
\omega^\times(L)=\omega^\times(M)=inf\{\omega^\times(\y)|\y \in L\}=inf\{\omega^\times(\y)|\y \in M\}
\end{displaymath}

\end{thm}
Furthermore from Theorem 2.8 we derive
\begin{equation}
\omega^\times(L)=max\big\{n,sup\{v\big|(2.6.2)\textrm{ } does\textrm{ } not\textrm{ } hold\textrm{ } for \textrm{ }R\}\big\}
\end{equation}

\begin{thm}
Let $L$ be a hyperplane of  $\R^n$ parameterized  by \\ $(x_1, \ldots,
x_{n-1}, a_1x_1+\ldots+a_{s-1}x_{s-1}+b)$ with $a_i\neq0,
i=1,\ldots,s-1$  \\
 $ \textrm{
 }\omega^\times(L)=max\big\{n,\frac{n}{s}\sigma(a_1,\ldots,a_{s-1},b)\big\}$
\end{thm}
\begin{proof}
(2.6.2) $ \Leftrightarrow \forall d_s>c_s$,$\exists T>0$ such that
$\forall t>T,$\\$\forall \w=(p_0,p_1,\ldots,p_n)
\in \Z^{n+1}_s$ with $p_i \neq 0$ for $1\leqslant i\leqslant s-1$ and $i=n$,
$max\Big(e^{-t_i}|p_i|,e^t\|RC_J(\w)\|\Big)\geqslant e^{-d_st}$.
Lemma 2.3 is applicable with
\begin{displaymath}
k=s,v=\dfrac{n+nc_s}{1-sc_s}, |x|=\|RC_J(\w)\|,z_i=p_i,E={(|x|,\z)} \subset \R^{n+1}
\end{displaymath}
(2.6.2) is equivalent to $\exists N>0 \textrm{ such that } \forall (x,\z) \in E \textrm{  with } \|\z\|>N $
\begin{equation}
\left\|  \begin{array}{c}
p_1+a_1p_n\\ p_2+a_2p_n \\ \vdots \\p_{s-1}+a_{s-1}p_n \\ p_0+bp_n \end{array}  \right\|
 \geqslant (p_1p_2\ldots p_{s-1}p_n)^{-v/n}
\end{equation}
By assuming $\|p_{i}+a_{i}p_n \|\leqslant 1$ for $1\leqslant i\leqslant s-1$, we know up to some constant (2.12) is the same as
\begin{displaymath}
\left\|  \begin{array}{c}
p_1+a_1p_n\\ p_2+a_2p_n \\ \vdots \\p_{s-1}+a_{s-1}p_n \\ p_0+bp_n \end{array}  \right\|
 \geqslant p_n^{-sv/n}
\end{displaymath}
Therefore by (2.11) $\omega^\times(L)=max\big\{n,\frac{n}{s}\sigma(a_1,\ldots,a_{s-1},b)\big\}$

\end{proof}
Theorem 1.1 is obtained by combining Theorem 2.9 and Theorem 2.10.

\bigskip

\section{A special class of hyperplanes and elementary approach}

We prove a special case of Theorem 1.1 via elementary methods here, namely,
for a hyperplane $L \subset \R^n$ parameterized as $(x_1, \ldots,
x_{n-1}, b)$ $ \textrm{
 }\omega^\times(L)=max\{n,n\sigma(b)\}$.\\
First, in the definition of $\omega^\times(\y)$
$\y=(x_1,\ldots,x_{n-1},b)$, if we set $\q=(0,0,\ldots,0,q)$ we see
that $\omega^\times(\y)\geqslant n\sigma(b)$, the coefficient $n$
arising from the denominator on the right hand  side of the equality
of (1.2). Therefore $\omega^\times(L)\geqslant max\{n,n\sigma(b)$.\\
 Second, $\prod_\times(\q)=\prod_{i=1\ q_i\neq 0}^{n}|q_i|\leqslant
 \|\q\|^n$\\
 so $\omega^\times(\y)\leqslant n\omega(\y)$,and $\omega^\times(L)\leqslant
 n\omega(L)=n\sigma(b)$.\\
 Together we have $\omega^\times(L)=max\{n,n\sigma(b)\}$.

\bigskip

\textbf{Acknowledgement.} The author is grateful to Professor
Kleinbock for helpful discussions.

\end{document}